\documentclass[times,authoryear,12pt]{elsarticle}
\date{\today}
\usepackage{natbib}
\usepackage{geometry,txfonts}
\usepackage[labelfont=bf]{caption}
\usepackage{amsmath,amsfonts,amssymb,amsthm,booktabs,color,epsfig,graphicx,url}
\urldef\myurl\url{www.ualberta.ca/~dwiens/}
\usepackage{appendix}
\usepackage[utf8]{inputenc}
\usepackage[english]{babel}
\newtheorem{lemma}{Lemma}
\newtheorem{remark}{Remark}

 \makeatletter
\def\ps@pprintTitle{%
  \let\@oddhead\@empty
  \let\@evenhead\@empty
  \def\@oddfoot{\reset@font\hfil\thepage\hfil}
  \let\@evenfoot\@oddfoot
}
\makeatother
 
\journal{U. Alberta preprint series}
\def\T{{ \mathrm{\scriptscriptstyle T} }}

\begin{document}
\begin{frontmatter}
\title{A Note on Minimax Robustness of Designs Against \linebreak Correlated or Heteroscedastic Responses}

\author[A1]{Douglas P. Wiens\corref{mycorrespondingauthor}}

\address[A1]{Mathematical \& Statistical Sciences,
	University of Alberta,
	Edmonton, Canada,  T6G 2G1
 \newline \newline \today}

\cortext[mycorrespondingauthor]{E-mail: \url{doug.wiens@ualberta.ca}.}

\begin{abstract}
We present a result according to which certain functions of covariance
matrices are maximized at scalar multiples of the identity matrix. This is used to
show that experimental designs that are optimal under an assumption of
independent, homoscedastic responses can be minimax robust, in broad classes of alternate covariance structures. In particular it can justify the common practice of disregarding possible dependence, or heteroscedasticity, at the design stage of an experiment.
\end{abstract}
\end{frontmatter}

\section{Introduction}
Experimental designs are typically derived, or chosen, assuming that the
responses will be independent and homoscedastic. As well as being simple,
this is almost necessary, unless an alternate covariance structure is somehow
known. This is frequently a complicating feature of design theory -- until a
design is constructed and implemented there are no data which can be used to
estimate the model. There is some consolation however if a
proposed design, optimal under an assumption of independence and 
homoscedasticity, is \textit{minimax} against a broad class of
alternate structures. By this we mean that the maximum loss in such a class is minimized by this design, which is thus robust against these alternatives. In this note we establish that any function of covariance matrices, possessing a natural monotonicity property, is maximized within such classes at a scalar multiple of the identity matrix. This can be paraphrased by
saying that the `least favourable' covariance structure is that of
independence and homoscedasticity. These are eventualities for which the proposed design is optimal; hence it is minimax.

\section{Main result\label{sec: main}}

Suppose that $\left\Vert \cdot \right\Vert _{M}$ is a matrix norm, induced by the vector norm $\left\Vert \cdot
\right\Vert _{V}$, i.e. 
\begin{equation}
\left\Vert C\right\Vert _{M}=\sup_{\left\Vert {x}%
\right\Vert _{V}\text{ }=1}\left\Vert{Cx}\right\Vert _{V}.
\label{norm}
\end{equation}%
We use the subscript `$M$' when referring to an arbitrary matrix norm, but
adopt special notation in the following cases:

\noindent(i) For the Euclidean norm $\left\Vert {x}\right\Vert _{V}=({{x}%
^{\T}{x}})^{1/2}$, the matrix norm is denoted $\left\Vert C\right\Vert _{E}$ and is
the spectral radius, i.e. the root of the maximum eigenvalue of  $C^{\T}C$. This is the maximum eigenvalue of $C$ if $C$ is a covariance matrix, i.e. is symmetric and positive semidefinite.

\noindent(ii) For the sup norm $\left\Vert {x}\right\Vert _{V}=\max_{i}\left\vert x_{i}\right\vert $, the matrix norm  $\left\Vert C\right\Vert _{\infty}$ is $\max_{i}\sum_{j} \left\vert c_{ij}\right\vert $, the maximum absolute row sum.

\noindent(iii) For the 1-norm $\left\Vert {x}\right\Vert _{V} = \sum_{i} \left\vert x_{i}\right\vert$, 
the matrix norm $\left\Vert C\right\Vert _{1}$ 
is $\max_{j}\sum_{i}\left\vert c_{ij}\right\vert $, the maximum absolute column sum. For symmetric matrices, $%
\left\Vert C\right\Vert _{1}=\left\Vert C\right\Vert _{\infty }$.

\noindent See Todd (1977, Chapter 3) for verifications of (i) -- (iii). 

We require the following properties\ of induced norms:

\noindent (P1) $\left\Vert I\right\Vert _{M}=1$.

\noindent (P2) For covariance matrices $C$, $\left\Vert C%
\right\Vert _{M}\geq \left\Vert C\right\Vert _{E}$.

Property P1 is immediate from (\ref{norm}). For P2, suppose that $\left\Vert C \right\Vert _{M}$ is induced by $\left\Vert \cdot \right\Vert _{V}$, and that $\lambda _{0}$ is the maximum eigenvalue of $C$, with
eigenvector $x_{0}$ normalized so that $\left\Vert x_{0}\right\Vert _{V}=1$. Then 
\begin{equation*}
\left\Vert C\right\Vert _{E}=\lambda _{0}=\left\Vert \lambda
_{0}x_{0}\right\Vert _{V}=\left\Vert Cx_{0}\right\Vert _{V}\leq
\sup_{\left\Vert x\right\Vert _{V}=1}\left\Vert Cx\right\Vert
_{V}=\left\Vert C\right\Vert _{M}.
\end{equation*}

Now suppose that the loss function in a statistical problem is $%
\mathcal{L}\left( C\right) $, where $C$ is an $%
N\times N$ covariance matrix and $\mathcal{L}\left( \mathbf{\cdot }\right) $
is non-decreasing in the Loewner ordering: 
\begin{equation}
A\preceq B\Rightarrow \mathcal{L}\left( 
A\right) \leq \mathcal{L}\left( B\right) .
\label{Loewner}
\end{equation}%
Here $A\preceq B$ means that $B-%
A\succeq 0$, i.e.\ is positive semidefinite.

\begin{lemma}For $\eta ^{2}>0$, covariance matrix $C$ and induced norm $%
\left\Vert C\right\Vert _{M}$, define
\begin{equation*}
\mathcal{C}_M =\{ C \mid C\succeq 
0\text{ and }\left\Vert C\right\Vert _{M}\leq \eta
^{2}\} .
\end{equation*}%
For the norm $\left\Vert \mathbf{\cdot }\right\Vert _{E}$ an equivalent
definition is%
\begin{equation}
\mathcal{C}_E =\{ C \mid 0\preceq 
C\preceq \eta ^{2}I_{N}\} .
\label{C_E}
\end{equation}%
Then:

\noindent (i) In any such class $\mathcal{C}_M$, $\max_{\mathcal{C}_M}\mathcal{L}
\left( C \right) =\mathcal{L}\left( \eta ^{2}I_{N}\right) $.

\noindent (ii) If $\mathcal {C^\prime }$ $\mathcal{\subseteq C}_{M}$ and 
$\eta^{2}I_{N}\in \mathcal {C^\prime }$, then 
$\max_{\mathcal {C^\prime }}\mathcal{L}
\left( C \right) =\mathcal{L}\left( \eta ^{2}I_{N}\right) $.
\end{lemma}

\begin{proof} We first establish the equivalence of (\ref{C_E}). If the condition there holds, then by Weyl's Monotonicity Theorem (Bhatia 1997, p. 63), all eigenvalues of $C$ are dominated by $\eta^2$, hence $\left\Vert C\right\Vert _{E}\leq \eta^{2}$. Conversely, if $\left\Vert C\right\Vert _{E}\leq \eta ^{2}$ then all eigenvalues of $C$ are dominated by $\eta^2$, hence all those of $\eta ^2 I_N - C$ are non-negative and so the condition in (\ref{C_E}) holds. 

By (\ref{Loewner}) and (\ref{C_E}),
$\max_{\mathcal{C}_E}\mathcal{L}\left( C \right) =\mathcal{L}\left( \eta ^{2}I_{N}\right) $.
Then by property P1 followed by P2, $\eta ^{2}I%
_{N}\in \mathcal C_{M}$ $\mathcal{\subseteq }$ $\mathcal C_{E}$, and so
the maximizer in the larger class is a member of the smaller class, hence 
\textit{a fortiori} the maximizer there. This proves (i). The proof of (ii) uses the same
idea -- the maximizer in the larger class $\mathcal C_{M}$ is a member of
the smaller class $\mathcal C^{\prime }$.
\end{proof}

\begin{remark}An interpretation of Lemma 1 is as follows.
Suppose that one has derived a technique under an assumption of
uncorrelated, homoscedastic errors, i.e.\ a covariance matrix $\sigma ^{2}%
I_{N}$, which is optimal in the sense of minimizing $\mathcal{L}
$, for any $\sigma ^{2}>0$. Now suppose one is concerned that the covariance
matrix might instead be a member $C$ of $\mathcal C_{M}$, and
that $\mathcal{L}$ is monotonic in the sense described above. Then the
technique minimizes $\max_{C_{M}}\mathcal{L}%
\left( C\right) =\mathcal{L}\left( \eta ^{2}I%
_{N}\right) $, i.e.\ is minimax in $\mathcal C_{M}$. 
\end{remark}

\begin{remark}In Remark 1 we implicitly assume that $\eta
^{2}\geq \sigma ^{2}$, else $\mathcal C_{M}$ does not contain $\sigma ^{2}%
I_{N}$. \ An argument for taking $\eta ^{2}>\sigma ^{2}$ arises
if one assumes homoscedasticity and writes $C=\sigma ^{2}%
P$, where $P$ is a correlation matrix. Then in $%
\mathcal C_{1}$, $\eta ^{2}\geq \left\Vert C\right\Vert _{1}$ $%
=\sigma ^{2}\left\Vert P\right\Vert _{1}\geq \sigma ^{2}$, with
the final inequality being an equality iff $P=I%
_{N} $. Thus take $\eta ^{2}>\sigma ^{2}$. Then an intuitive explanation of
the lemma is that in determining a least favourable covariance structure,
one can alter the correlations in some manner which increases $\left\Vert 
C\right\Vert _{M}$, or one can merely increase the variances.
The answer is that one should always just increase the variances.
\end{remark}

\begin{remark}A version of Lemma 1 was used by Wiens and Zhou
(2008) in a maximization problem related to the planning of field
experiments. It was rediscovered by Welsh and Wiens (2013) in a study of
model-based sampling procedures. This note seems to be the first systematic
study of the design implications of the lemma.
\end{remark}

\section{Applications\label{sec: examples}}

\subsection{Experimental design in the linear model\label{sec: design}}

Consider the linear model $y=X\theta+%
\varepsilon$. Suppose that the random errors $\varepsilon$ have covariance
matrix $C \in \mathcal C_{M}$. If $C$ is known then the
`Best Linear Unbiased Estimate' is $\hat{\theta}_{\text{\textsc{%
blue}}}=$ $\left( X^{\T }C^{-1}X%
\right) ^{-1}X^{\T }C^{-1}y$ and
there is an extensive design literature - see Dette et al. (2015) for a review. In the more common case that the
covariances are only vaguely known, or perhaps only suspected, it is more
usual to use the ordinary least squares estimate $\hat{\theta}_{%
\text{\textsc{ols}}}$, design as though the errors are uncorrelated, and
hope for the best. An implication of the results of this section is that, in
a minimax sense, that approach can be sensible.

In the classical `alphabetic' design problems, one seeks to minimize a
function $\Phi $ of the covariance matrix of the regression estimates. Let $%
\xi _{0}$ be the minimizing design, under the possibly erroneous assumption of uncorrelated, homoscedastic errors. Assume\ that under $\xi _{0}\ $the moment matrix $%
X^{\T}X$ is non-singular. Then the covariance
matrix of $\hat{\theta}_{\text{\textsc{ols}}}$ is 
\begin{equation}
\textsc{cov}( \hat{\theta}_{\textsc{ols}} \mid C)
=\left( X^{\T }X\right) ^{-1}X%
^{\T }CX\left( X^{\T}X%
\right) ^{-1}.
\label{cov}
\end{equation}
Suppose that $0\preceq C_{1}\preceq C_{2}$%
, so that $C_{2}-C_{1}=A^{\T }A$ for some $A$. Then
\begin{equation*} 
\textsc{cov}( \hat{\theta}_{\textsc{ols}} \mid C_{2}) -  \textsc{cov}( \hat{\theta}_{\textsc{ols}} \mid C_{1}) = B^{\T}B \succeq 0,
\end{equation*}
for $B$ $=AX\left( %
X^{\T }X\right) ^{-1}$, hence 
$
\textsc{cov}( \hat{\theta}_{\textsc{ols}} \mid C_{1}) \preceq \textsc{cov}( \hat{\theta}_{\textsc{ols}} \mid C_{2}) .
$
It follows that if $\Phi $ is non-decreasing in the Loewner ordering, then 
$
\mathcal{L}( C) =\Phi \{ \textsc{cov}( \hat{\theta}_{\textsc{ols}} \mid C)\}
$
is also non-decreasing and the conclusions of the lemma hold. Then as at
Remark 1, $\xi _{0}$ is a minimax design -- it minimizes the maximum loss as
the covariance structure varies over $\mathcal C_{M}$.

Again by Weyl's Monotonicity Theorem, if $0\preceq \Sigma_{1}\preceq 
\Sigma_{2}$ then the $i$th largest eigenvalue $\lambda
_{i} $ of $\Sigma_{2}$ dominates that of $\Sigma%
_{1}$, for all $i$. It follows that $\Phi $ is non-decreasing in the Loewner
ordering if:

\noindent (i) $\Phi \left( \Sigma\right) =tr\left( \Sigma\right) =\sum_{i} \lambda _{i}\left( \Sigma\right) $,
corresponding to `A-optimality';

\noindent (ii) $\Phi \left( \Sigma\right) =det\left( \Sigma\right) =\prod_{i} \lambda _{i}\left( \Sigma\right) $, corresponding to `D-optimality';

\noindent (iii) $\Phi \left( \Sigma\right) =\max_{i}\lambda
_{i}\left( \Sigma\right) $, corresponding to `E-optimality';

\noindent (iv) $\Phi \left( \Sigma\right) =tr\left( 
L\Sigma \right) $ for $L\succeq 0$, corresponding to `L-optimality' and including `I-optimality' --
minimizing the integrated variance of the predictions. Thus the designs optimal under any of these criteria are minimax.

\begin{subsubsection}{\textsc{ma(1)} (moving average of order 1) errors.}

As a particular case, assume first that the random errors are
homoscedastic but are possibly serially correlated, following an \textsc{ma(1)} model with
\textsc{corr}$(\varepsilon _{i},\varepsilon _{j}) =\rho I\left( \left\vert i-j\right\vert =1\right) $ and with $\left\vert \rho \right\vert \leq \rho
_{\max }$. Then under this structure $C$ varies
over the subclass $\mathcal C^{\prime }$ of $\mathcal C_{\infty }$ defined
by $c_{ij}=0$ if $\left\vert i-j\right\vert >1$ and $\left\Vert {C}\right\Vert _{\infty }\leq \sigma ^{2}\left( 1+2\rho _{\max }\right) $, which we define to be $\eta ^{2}$. Since $\eta ^{2}I_{N} \in \mathcal C^{\prime }$, part (ii) of the lemma applies and yields that  $\xi _{0}$ is a minimax design in $\mathcal{C}^{\prime }$ and with respect to any of the alphabetic
criteria above. If the errors are instead heteroscedastic, then $\sigma^2$ is replaced by the maximum of the variances.
\end{subsubsection}

\begin{subsubsection} {\textsc{ar(1)} (autoregressive of order 1) errors.}

It is known that the eigenvalues of an \textsc{ar(1)} autocorrelation matrix with
autocorrelation parameter $\rho $ are bounded, and that the maximum
eigenvalue $\lambda \left( \rho \right) $ has $\lambda^{\ast }=$ $%
\max_{\rho }\lambda\left( \rho \right) >\lambda \left( 0\right) =1$%
. See for instance Trench (1999, p.\ 182). Then, again under
homoscedasticity, the covariance matrix $C$ has $%
\left\Vert C\right\Vert _{E}\leq \sigma ^{2}\lambda ^{\ast }$, and a design optimal when $\rho =0$ is minimax
in the subclass of $\mathcal C_{E}$ defined by the autocorrelation structure and 
$\eta ^{2} = \sigma ^{2}\lambda ^{\ast} $.
\end{subsubsection}

\subsection{Designs robust against model misspecifications}

Working in finite design spaces $\chi $ and with $p$-dimensional regressors $%
f\left( x\right) $, Wiens (2018) derived minimax
designs for possibly misspecified regression models 
\begin{equation}
Y\left( x\right) =f^{\T }\left( x%
\right) \theta+\psi \left( x\right) +\varepsilon ,
\label{model}
\end{equation}%
with the unknown contaminant $\psi $ ranging over a class $\Psi $ and
satisfying, for identifiability of $\theta$, the orthogonality condition 
\begin{equation}
\sum_{x\in \chi }f\left( x\right)
\psi \left( x\right) =0_{p\times 1}.  \label{orthogonality}
\end{equation}%
For designs $\xi$ placing mass $\xi _{i}$ on $x%
_{i}\in \chi $, he took  $\hat{\theta}=\hat{\theta}_{\text{%
\textsc{ols}}}$, 
\begin{subequations}
\label{loss functions}
\begin{eqnarray}
\mathcal{I}\left( \psi ,\xi\right) &=&\sum_{x\in \chi }
E[f^{\T }\left(x\right) \hat{\theta}-E\{Y\left( x\right) \}] ^{2},  \label{I-loss} \\
\mathcal{D}\left(\psi ,\xi\right) &=&[ \det E \{
(\hat{\theta}-\theta) ( 
\hat{\theta}-\theta) ^{\T}\}
] ^{1/p},  \label{D-loss}
\end{eqnarray}%
and found designs minimizing the maximum, over $\psi $, of these loss
functions. The random errors $\varepsilon _{i}$ were assumed to be
i.i.d.; now suppose that they instead have covariance matrix $C%
\in \mathcal C_{M}$.

Consider first (\ref{I-loss}). Using (\ref{orthogonality}), and with $%
d_{\psi }=\{ E( \hat{\theta}) -%
\theta\} $, which does not depend on the covariance
structure, this decomposes as 
\end{subequations}
\begin{equation}
\mathcal{I}\left( \psi ,\xi,C\right) =\sum_{%
x\in \chi }f^{\T}\left( x\right) 
{\textsc{cov}}(\hat{\theta} \mid C) f%
\left( x\right) +\sum_{x\in \chi }\left\{ 
f^{\T }\left( x\right) d_{\psi }%
d_{\psi }^{\T }f\left( x\right)
+\psi ^{2}\left( x\right) \right\} .  \label{I(C)}
\end{equation}%
The first sum above does not depend on $\psi $; the second depends on $\psi $
but not on the covariance structure. Then an extended minimax problem is to
find designs $\xi$ minimizing 
\begin{eqnarray*}
\max_{\psi ,C}\mathcal{I}\left( \psi ,\xi,%
C\right)  = \max_{C\in \mathcal C_{M}}\sum_{%
x\in \chi }f^{\T }\left( x\right) 
\text{\textsc{cov}}( \hat{\theta} \mid C) f%
\left( x\right) + \max_{\psi \in \Psi }\sum_{x\in \chi }\left\{ f%
^{\T}\left( x\right) d_{\psi }d%
_{\psi }^{\T }f\left( x\right) +\psi
^{2}\left( x\right) \right\} .
\end{eqnarray*}%
As in \S \ref{sec: design}, and taking $L=\sum_{x%
\in \chi }f\left( x\right) f^{\T
}\left( x\right) $ in (iv) of that section, for $C%
\in \mathcal C_{M}$ the first sum is maximized by a multiple $\eta ^{2}$ of
the identity matrix, and then the remainder of the minimax problem is that
which was solved in Wiens (2018). The minimax designs, termed `I-robust'
designs, obtained there do not depend on the value of $\eta ^{2}$, and so
enjoy the extended property of minimizing $\max_{\psi ,C}%
\mathcal{I}\left( \psi ,\xi,C\right) $ for $%
C\in \mathcal C_{M}$.

Now consider (\ref{D-loss}). The analogue of (\ref{I(C)}) is 
\begin{eqnarray*}
\mathcal{D}\left( \psi ,\xi,C\right) =[ \det
\{ \text{\textsc{cov}}( \hat{\theta} \mid C) +%
d_{\psi }d_{\psi }^{\T }\} ]
^{1/p}.
\end{eqnarray*}
Since \textsc{cov}$( \hat{\theta} \mid C ) +d%
_{\psi }d_{\psi }^{\T}$, hence its determinant, is
non-decreasing in the Loewner ordering, $\mathcal{D}\left( \psi ,\xi,C\right) $ is maximized, for $C\in \mathcal{C%
}_{M}$, by a multiple of the identity matrix. The rest of the argument is
identical to that in the preceding paragraph, and so the `D-robust' designs
obtained in Wiens (2018) also minimize $\max_{\psi ,C}\mathcal{D%
}\left( \psi ,\xi,C\right) $ for $C%
\in \mathcal C_{M}$.

\begin{remark}Results in the same vein as those above have
been obtained in cases which do not seem to be covered by Lemma 1. For
instance Wiens and Zhou (1996) sought minimax designs for the
misspecification model (\ref{model}), under conditions on the spectral
density of the error process. They state that \textquotedblleft {...
a design which is asymptotically (minimax) optimal for uncorrelated errors
retains its optimality under autocorrelation if the design points are a
random sample, or a random permutation, of points ...}\textquotedblright ,\
with details in their Theorems 2.4 and 2.5.
\end{remark}

\subsection{Designs for nonlinear regression models}
In the nonlinear regression model $y=f\left( x;\theta \right) +\varepsilon $, 
the goal of a design is typically the minimization of some function of the
covariance matrix of $\hat{\theta}$, after a linearization of the model. When the errors have
covariance $C$ the target of this minimization continues to be given by the
right hand side of (\ref{cov}), but with the matrix $X$ replaced by the gradient $%
F\left( \theta \right) =\{ \partial f\left( x_{i};\theta \right)
/\partial \theta _{j}\}_{i,j} $, and with $F\left( \theta \right) $ evaluated
at an initial parameter estimate, or at a previous estimate in an
iterative estimation scheme. Then the results of \S \ref{sec: design} continue to hold, and 
a design optimal for the loss functions given there, under independence and homoscedasticity, is minimax. A
caveat, however, is that optimal design theory for nonlinear models is much
less well-developed, and more model specific, than that for linear models.
See Bates and Watts (1988) and Hamilton and Watts (1985) for discussions. 

\section*{Acknowledgements}
This work was carried out with the support of the Natural Sciences and
Engineering Research Council of Canada. It has benefited from the helpful comments of two anonymous referees.


\end{document}